\documentclass[12pt,a4paper]{amsart}
\usepackage{setspace}
\usepackage{hyperref}
\usepackage{amsfonts}
\usepackage[left=40mm,top=40mm,right=38mm,bottom=44mm]{geometry}
\usepackage{amssymb}
\usepackage{latexsym}
\usepackage{amsmath}
\usepackage{mathrsfs}
\usepackage{epigraph}
\usepackage[all,2cell]{xy}
 \UseAllTwocells \SilentMatrices
\usepackage{ifthen}
\usepackage{comment}

\usepackage{array,graphics,color}

\newenvironment{gap}
  {\color{blue}}%
  {}%

\newenvironment{nap} 
  {\color{red}}%
  {}%

\newtheorem{defn}{Definition}[section]
\newtheorem{lem}[defn]{Lemma}
\newtheorem{thm}[defn]{Theorem}
\newtheorem{cor}[defn]{Corollary}
\newtheorem{ex}[defn]{Example}

\newtheorem{prop}[defn]{Proposition}
\newtheorem{rem}[defn]{Remark}

\def\p{\mathcal{P}}

\DeclareMathOperator{\Ker}{Ker}

\DeclareMathOperator{\Dom}{Dom}
\DeclareMathOperator{\Ran}{Ran}

\DeclareMathOperator{\loc}{loc}
\DeclareMathOperator{\tr}{tr}

\newcommand{\en}{E\ensuremath{{_0}}}

\title[Cohomology for spatial super-product systems]{Cohomology for spatial super-product systems 
}
\author[O. Margetts]{Oliver T. Margetts}
\address{Department of Mathematics and Statistics,
 Fylde College,
Lancaster University, Lancaster LA1 4YF, U.K.}
\email{o.margetts@lancaster.ac.uk}

\author[R. Srinivasan]{R. Srinivasan}
\address{Chennai Mathematical Institute, H1, SIPCOT IT Park, Kelambakkam, Siruseri 603103, India.}
\email{vasanth@cmi.ac.in}


\begin{document}


\newcommand{\norm}[1]{\ensuremath{\left\|#1\right\|}}
\newcommand{\ip}[1]{\ensuremath{\left\langle#1\right\rangle}}
\newcommand{\dist}{\hbox{dist}}
\newcommand{\bra}[1]{\ensuremath{\left\langle#1\right|}}
\newcommand{\ket}[1]{\ensuremath{\left|#1\right\rangle}}
\newcommand{\lin}{\ensuremath{\mathrm{Span}}}
\renewcommand{\ker}{\ensuremath{\mathrm{ker}}}
\newcommand{\ran}{\ensuremath{\mathrm{Ran}}}
\newcommand{\dom}{\ensuremath{\mathrm{Dom}}}
\newcommand{\supp}{\ensuremath{\mathrm{supp}}}
\newcommand{\id}{\hbox{id}}

\newcommand{\overtimes}{\ensuremath{\overline{\otimes}}}
\newcommand{\undertimes}{\ensuremath{\underline{\otimes}}}
\newcommand{\opower}[1]{\ensuremath{^{\otimes #1}}}

\newcommand{\N}{\ensuremath{\mathbb{N}}} 
\newcommand{\Z}{\ensuremath{\mathbb{Z}}} 
\newcommand{\Q}{\ensuremath{\mathbb{Q}}} 
\newcommand{\R}{\ensuremath{\mathbb{R}}} 
\newcommand{\C}{\ensuremath{\mathbb{C}}} 
\newcommand{\F}{\ensuremath{\mathcal{F}}} 
\newcommand{\B}{\ensuremath{\mathcal{B}}}
\newcommand{\D}{\ensuremath{\mathcal{D}}} 
\newcommand{\E}{\ensuremath{\mathcal{E}}} 
\newcommand{\W}{\ensuremath{\mathcal{W}}}
\newcommand{\V}{\ensuremath{\mathcal{V}}}
\renewcommand{\H}{\ensuremath{\mathcal{H}}}
\newcommand{\K}{\ensuremath{\mathsf{K}}}
\newcommand{\h}{\ensuremath{\mathrm{h}}}
\renewcommand{\k}{\ensuremath{\mathrm{k}}}
\newcommand{\J}{\ensuremath{\mathscr{J}}}
\newcommand{\A}{\ensuremath{\mathcal{A}}}
\renewcommand{\L}{\ensuremath{\mathcal{L}}}
\newcommand{\n}{\ensuremath{\mathrm{N}}} 
\newcommand{\m}{\ensuremath{\mathrm{M}}} 

\newcommand{\munit}{\ensuremath{\mu\hbox{nit~}}}
\newcommand{\munits}{\ensuremath{\mu\hbox{nits~}}}
\newcommand{\unitset}[1]{\ensuremath{\mathcal{U}_{#1}}}
\newcommand{\unitspace}[1]{\ensuremath{H(\mathcal{U}_{#1})}}
\newcommand{\munitset}[1]{\ensuremath{\mathcal{U}_{#1,#1'}}}
\newcommand{\munitspace}[1]{\ensuremath{H(\mathcal{U}_{#1,#1'})}}
\newcommand{\gaugespace}[1]{\ensuremath{H(G(#1))}}
\newcommand{\ind}{\ensuremath{\mathrm{Ind}}}

\newcommand{\expectation}{\ensuremath{\mathbb{E}}}
\newcommand{\Exp}{\ensuremath{\mathrm{Exp}}}
\newcommand{\Log}{\ensuremath{\mathrm{Log}}}

\newcommand{\alg}{\ensuremath{\mathrm{A}}}
\newcommand{\factor}{\ensuremath{\mathrm{M}}}
\newcommand{\oneinfty}{I$_\infty~$}
\newcommand{\twoone}{II$_1~$}
\newcommand{\twoinfty}{II$_\infty~$}
\newcommand{\three}[1]{III$_{#1}~$}
\newcommand{\hyperfinite}{\ensuremath{\mathcal{R}}}
\newcommand{\semiflowalg}{\ensuremath{\mathcal{A}}}



\begin{abstract}
We introduce a cohomology theory for spatial super- product systems and compute the $2-$cocycles  for some basic examples called as Clifford super-product systems, thereby distinguish them up to isomorphism. This consequently proves that a family of \en-semigroups on type III factors, which we call as CAR flows, are non-cocycle-conjugate for different ranks.  Similar results follows for the even CAR flows as well. We also compute the automorphsim group of the Clifford super-product systems.
\end{abstract}

 \maketitle


\section{Introduction}

The super-product system of Hilbert spaces is a generalisation of Arveson's product system of Hilbert spaces. These were originally defined in \cite{MS1}, but the idea was being discussed before that. Apart from being an interesting mathematical object on its own,  super-product systems arises naturally as an invariant associated to \en-semigroups on factors, as shown in \cite{MS2}. In a sense this generalizes Arveson's association of product systems to \en-semigroups on type I factors, but this association is not one-one in the non-type I case.  We recall the basic definitions and the association of super-product system to an \en-semigroup in this introductory section.

\begin{defn}
 \label{superproductsystem}
 A super-product system of Hilbert spaces is a one parameter family
 of separable Hilbert spaces $\{H_t: t > 0 \}$, together with isometries $$U_{s,t} : H_s \otimes H_t ~\mapsto H_{s+t}~ \mbox{for}~ s, t \in (0,\infty),$$
 satisfying the following two axioms of associativity and measurability.

 \medskip
 \noindent(i) (Associativity) For any $s_1, s_2, s_3 \in (0,\infty)$
 $$U_{s_1, s_2 + s_3}( 1_{H_{s_1}} \otimes U_{s_2 ,
 s_3})=  U_{s_1+ s_2 , s_3}( U_{s_1 ,
 s_2} \otimes 1_{H_{s_3}}).$$

 \noindent (ii) (Measurability)  The space $\H=\{(t, \xi_t): t \in (0,\infty), \xi_t \in H_t\}$ is equipped with a structure of standard Borel space that is compatible with the projection $p:\H\mapsto (0,\infty)$ given by $p((t, \xi_t)=t$, tensor products and the inner products (see  \cite[Remark 3.1.2]{Arv}). 
\end{defn}

A super-product system is an Arveson product system if the isometries $U_{s,t}$ are unitaries and further the condition of local triviality is satisfied, that is there exists a single separable Hilbert space $H$ satisfying $\H \cong (0,\infty) \times H$ as measure spaces (see \cite[Remark 3.1.2]{Arv}).



\begin{defn}\label{spdct-iso} By an isomorphism between two super-product systems $(H^1_t, U^1_{s,t})$ and $(H^2_t, U^2_{s,t})$ we mean an isomorphism of Borel spaces $V:\H^1 \mapsto \H^2$ whose restriction to each fiber provides a unitary operator $V_t:H^1_t\mapsto H^2_t$ satisfying \begin{equation*}\label{prodiso}V_{s+t}U^1_{s,t}= U_{s,t}^2 (V_s \otimes V_t).\end{equation*} 
\end{defn}

Basic examples of product systems are exponential product systems and antisymmetric product systems, whose subsystems provides our basic examples of proper super-product systems. These are discussed in the next section.

The study of \en-semigroups was initiated by R. T. Powers (see \cite{Pow}), and Arveson made many important contributions through a sequence of papers in late 80s and 90s (see \cite{Arv}). Arveson showed that \en-semigroups on type I factors are completely classified by their associated product systems.  This gives a rough division of \en-semigroups into three types, based on the existence of units, namely I, II and III. The type I \en-semigroups on type I factors are cocycle conjugate to the CCR flows (\cite{Arv}), but there are uncountably many \en-semigroups of types II and III (\cite{pow2},  \cite{T1},  \cite{pdct}, \cite{genccr}, \cite{toepcar} \cite{lieb}) on type I factors. 


We say a von Neumann algebra $\m$ is in standard form if $\m\subseteq B(H)$ has a cyclic and separating vector $\Omega \in H$. 
By replacing a conjugate \en-semigroup if needed, without loss of generality, we can always assume that an \en-semigroup is acting on a von Neumann algebra in a standard form, thanks to \cite[Lemma 2.4 ]{MS2}.

\begin{defn}
An \en-semigroup on a von Neumann algebra $\m$ is a semigroup $\{\alpha_t: t\geq0\}$ of normal, unital *-endomorphisms of $\m$ satisfying
\begin{itemize}
 \item[(i)] $\alpha_0=id$,
 \item[(ii)] $\alpha_t(\factor)\neq\factor$ for all $t>0$,
 \item[(iii)] $t\mapsto \rho(\alpha_t(x))$ is continuous for all $x\in\factor$, $\rho\in\factor_*$.
\end{itemize}
\end{defn}

\begin{defn} A cocycle for an \en-semigroup $\alpha$ on $\m$ is a strongly continuous family of unitaries $\{U_t: t\geq 0\}\subseteq \m$ satisfying $U_s\alpha_s(U_t)=U_{s+t}$ for all $s,t\geq0$.
\end{defn}

For a cocycle $\{U_t: t\geq 0\}$, we automatically have $U_0=1$. Furthermore the family of endomorphisms $\alpha_t^U(x):=U_t\alpha_t(x)U_t^*$ defines an \en-semigroup. This leads to the equivalence of cocycle conjugacy on \en-semigroups. The following definition is called as spatial conjugacy, but we can assume any two conjugate \en-semigroups in standard form are spatially conjugate, thanks to  \cite[Lemma 2.4 ]{MS2}. So we take the following as the definition of conjugacy.

\begin{defn}
Two \en-semigroups $\alpha$ and $\beta$, acting standardly on $\m\subseteq B(H_1)$ and $ \n\subseteq B(H_2)$ respectively, are said to be conjugate if there exists a unitary $U:H_1\mapsto H_2$ satisfying
\begin{itemize}
 \item[(i)] $U\m U^*=\n$,
 \item[(ii)] $\beta_t(x)=U\alpha_t(U^*xU)U^*$ for all $t\geq0$, $x\in\n$,
\end{itemize}
\end{defn}

\begin{defn}\label{conjugacy def} Let $\alpha$ and $\beta$ be \en-semigroups on von Neumann algebras $\m$ and $\n$. 
Then $\alpha$ and $\beta$ are said to be \emph{cocycle conjugate} if there exists a cocycle $\{U_t: t\geq 0\}$ for $\alpha$ such that $\beta$ is conjugate to $\alpha^U$. 
\end{defn}

Given an \en-semigroup $\theta$ on $B(H)$, the associated Arveson product system $(H_t, U_{s,t})$ is defined by $$H_t =\{T \in B(H); \alpha_t(X) T = TX, ~ \forall ~X \in
B(H)\}$$ with inner product $\langle T, S \rangle 1_H= S^*T$ and $U_{s,t}(T\otimes S) = TS$ for $T \in H_s, S \in H_t$. 
 
Let $\alpha$ be an \en-semigroup on a factor $\m$ with cyclic and separating vector $\Omega$ and let $J_\Omega$ be the modular conjugation associated with $\Omega$ by the Tomita-Takesaki theory. We can define a complementary \en-semigroup $\alpha'$ on $\m'$ by setting
 $$\alpha'_t(x')=J_\Omega\alpha_t(J_\Omega x'J_\Omega)J_\Omega\qquad (x'\in\m').$$  The complementary \en-semigroup is determined up to conjugacy, thanks to the following Proposition (for proof see  \cite[Proposition 3.1 ]{MS2}).

 \begin{prop}\label{doubly intertwining unitary shizzle}
 Let $\m$ and $\n$ be von Neumann algebras acting standardly with respective cyclic and separating vectors $\Omega_1\in H_1$, $\Omega_2\in H_2$. If the \en-semigroups $\alpha$ on $\m$, and $\beta$ on $\n$ are cocycle conjugate, then $\alpha'$ and $\beta'$, defined with respect to $\Omega_1$ and $\Omega_2$ respectively,  are cocycle conjugate. Moreover, if $\alpha$ and $\beta$ are conjugate, then $\alpha'$ and $\beta'$ are spatially conjugate and the implementing unitary can be chosen so that it also intertwines $\alpha$ and $\beta$.
 \end{prop}
 
A super-product system is associated as an invariant to an \en-semigroup on a factor through the following theorem (see \cite[Theorem 3.4]{MS2}).

\begin{thm}\label{sps theorem}
  Let $\m\subseteq B(H)$ be a factor in standard form and $\alpha$ an \en-semigroup on $\m$. For each $t>0$, let
 $$H^{\alpha}_t=\{ X\in B(H):\forall_{m\in\m, m'\in\m'}~\alpha_t(m)X=Xm, ~\alpha'_t(m')X=Xm' \}. $$
Then $H^{\alpha}=\{H^{\alpha}_t: t>0\}$ is a super-product system with respect to the family of isometries $U_{s,t}\left(X\otimes Y \right)=XY$.

Let $\alpha$ and $\beta$ be \en-semigroups acting on respective factors $\m$ and $\n$ in standard form. If $\alpha$ and $\beta$ are cocycle conjugate then $H^{\alpha}$ and $H^{\beta}$ are isomorphic.
 \end{thm}
 
In \cite{MS1}, super-product systems of a family of \en-semigroups on type II$_1$ factors, called as Clifford flows and even Clifford flows, were computed. Though the Clifford flows and even Clifford were shown to be non-cocycle-conjugate for different ranks, the proof was indirect, using boundary representations and a theory of $C^*-$semiflows. The associated super-product systems were not shown to be non-isomorphic when their ranks are different. In this paper we show those super-product systems to be non-isomorphic for different ranks (see Section \ref{2-addits computation}), which gives a direct proof regarding the non-cocycle-conjugacy of Clifford flows and even Clifford flows. 

In \cite{MS2}, using CCR representations, uncountable families consisting of  mutually non-cocycle-conjugate \en-semigroups were constructed on every type III$_\lambda$ factors, for $\lambda \in (0,1]$. These are called as CCR flows on type III factors. Similar to CCR representations, it is also possible to produce families of \en-semigroups using CAR representations. But useful invariants are not yet found to distinguish them up to cocycle conjugacy. This difficulty is mainly because the gauge group turns out to be trivial for \en-semigroups constructed through CAR representations, unlike the CCR representations. 

In this paper we  distinguish CAR flows on type III factors when their ranks are different, 
by showing the associated super-product systems are non-isomorphic (see Section \ref{2-addits computation}).  It remains open to classify  among CAR flows with same rank, similar to the classification done for CCR flows in section 8 in \cite{MS2}. 

\section{Preliminaries}\label{pre}

$\N$ denotes the set of natural numbers, and we set $\N_0=\N\cup \{0\},$ $\overline{\N}=\N\cup \{\infty\}$. We will only deal with complex Hilbert spaces in this paper. The inner product is always conjugate linear in the first variable and linear in the second variable. $L^2(S,\k)$ denote the square integrable functions from $S$ taking values in a complex separable Hilbert space  $\k$. $L^2_{\loc}(S,\k)$ denotes the functions which are square integrable on compact subsets.  Throughout this paper we denote by $(T_t)_{t\geq 0}$ the right shift semigroup on $L^2((0,\infty),\k)$ 
defined by 
\begin{eqnarray*}(T_tf)(s) & = & 0, \quad s<t,\\
& = & f(s-t), \quad s \geq t,
\end{eqnarray*} 
for $f \in L^2((0,\infty),\k)$.
 
We  will be dealing with antisymmetric tensors in all the examples considered in this paper, since they are the ones related to  our examples of \en-semigroups. But examples of super-product systems in this paper and the facts proven about them can be extended to symmetric case as well, by making analogous changes. 

Let $K$ be a complex separable Hilbert space.  
Let $\Gamma(K):=\bigoplus_{n=0}^\infty{(K)^{\wedge n}}$ be the antisymmetric Fock space over $K$, 
where $K^{\wedge 0}$ is $\C$.
For any $f\in K$ the Fermionic creation operator $a^*(f)$ is the bounded operator defined by the linear extension of
$$a^*(f)\xi=\left\{\begin{array}{ll} f & \hbox{if}~\xi=\Omega, \\ f\wedge \xi &  \hbox{if}~ \xi\perp\Omega, \end{array}\right.$$ where $\Omega$ is the vacuum vector. 
The annihilation operator is defined by the adjoint $a(f)=a^*(f)^*$. The creation and annihilation operators satisfy the well-known anti-commutation relations (see  Equations (\ref{a-relations}) in Section \ref{CAR}) and generate $B(\Gamma(K))$. For an isometry $U:K_1 \mapsto K_2$, the second quantization $\Gamma(U): \Gamma(K_1)\mapsto \Gamma(K_2)$ is the isometry defined by the extension of  $$\Gamma(U)(x_1\wedge x_2\wedge\cdots \wedge x_n) = Ux_1\wedge Ux_2\wedge\cdots \wedge Ux_n.$$ 
For $s \in (0,\infty), t \in (0,\infty]$, define $$U_{s,t}:\Gamma(L^2((0,s), \k) \otimes \Gamma(L^2((0,t), \k) \mapsto \Gamma(L^2((0,s+t), \k)$$ as the unitary the extension of 
$$(\xi_1 \wedge\xi_2 \wedge \cdots \wedge \xi_{m}) \otimes (\eta_1 \wedge\eta_2  \wedge \cdots \wedge \eta_{n}) \mapsto T_s \eta_1 \wedge T_s \eta_2\wedge  \cdots \wedge T_s \eta_{n} \wedge \xi_1 \wedge\xi_2  \wedge \cdots \wedge \xi_{m}.$$ To avoid messy notations, we  will continue to use $U_{s,t}$ for isometric restrictions of $U_{s,t}$ to subspaces as well and for different $\k$.

\begin{ex}\label{CARflow} The \textit{CAR flow} of index $n=\dim \k\in \overline{\N}$  is the \en-semigroup $\theta^n=\{\theta^n_t:t\geq 0\}$ acting on $B(\Gamma(L^2((0,\infty),\k)))$ defined by the extension of $$\theta^n_t(a(f)):=a(T_tf), ~~ f \in L^2((0,\infty),\k).$$ 
\end{ex}

\begin{ex}\label{pdctI}  
The Arveson product system associated with CAR flow of index $\dim(\k)$ is described by the \textit{CAR product system} 
$(H^{\k}(t), U_{s,t})$, where $H^{\k}(t) =\Gamma(L^2((0,t), \k)$. 
 (see \cite[Remark 5.13]{MS1}.)
\end{ex}

The following example gives several families of proper super-product systems. The basic examples  we are concerned with are Clifford super-product systems and CAR super-product systems. 

 \begin{ex}\label{subsemigroup example}
For an additive sub-semigroup 
$G$ of $\N_0$, define
$$ H^{\k}_G(t):=\bigoplus_{n\in G} L^2([0,t], \k)^{\wedge n}  \subseteq  H^{\k} (t);$$ 
Then $\left(H^{\k}_G(t),U_{s,t}\right)$ forms a  super-product system.

When $G=2 \N_0$, we call them as the \textit{Clifford super-product systems}, as they are ones associated with Clifford flows (also with the even Clifford flows) on hyperfinite II$_1$ factor (see \cite[Corollary 8.13 ]{MS1}).
 \end{ex}

 \begin{ex}\label{subsemigroup example2} We  can take tensor products of super-product systems, where both the Hilbert spaces and  isometries are tensored accordingly, to produce new super-product systems.  
 Define
$$ E^{\k}_{2\N_0}(t)  :=\bigoplus_{n_1+n_2 \in 2 \N_0} L^2([0,t],\k)^{\wedge n_1}\otimes L^2([0,t],\k)^{\wedge n_2} \subseteq  H^\k(t)  \otimes   H^\k(t).$$  Denote $U^2{s,t} = U_{s,t} \otimes U_{s,t}$ (and also its restrictions).
 Then $(E^{\k}_{2 \N_0}(t),U^2_{s,t})$ forms a  super-product system.

By restricting the natural isomorphism between the product systems $ \left(H^\k(t) \otimes H^\k(t),  U^2{s,t}\right)$ and $\left(H^{\k \oplus \k}, U_{s,t}\right)$, it is easy to see that  $$(E^{\k}_{2 \N_0}(t),U^2_{s,t})   \cong (H^{\k\oplus \k }_{2 \N_0}(t),U_{s,t}).$$ 
 \end{ex}

\section{Types of super-product systems}\label{types}

Imitating the definition of types for product systems, super-product systems can also be broadly divided into three types. 
We further divide type II super-product systems based on the existence of 2-addits, which can be further refined by considering the $n-$th cohomology.  

\begin{defn} A unit for a super-product system $(H_t, U_{s,t})$ is a measurable section $u=\{u_t :u_t \in H_t\}_{t\geq 0}$ satisfying $U_{s,t}(u_s\otimes u_t)=u_{s+t}~\forall ~ s,t \in (0,\infty).$ \end{defn}

A super-product system is called spatial if it admits a unit. In a spatial super-product system we usually fix a special unit called as the canonical unit, denoted by $\Omega=\{\Omega_t\in H_t\}$. 
An exponential unit is a unit $u$ satisfying $\ip{u_t,\Omega_t}=1$. 
We denote 
the collection of all exponential units by $\mathfrak{U}_\Omega(H)$. 
For every unit $u$ there exists a scalar $\lambda \in \C$ such that  $\{e^{-\lambda t}u_t \}\in \mathfrak{U}_\Omega(H)$.

\begin{defn} An addit for a spatial super-product system $(H_t, U_{s,t})$, with respect to a canonical unit $\Omega$, is a measurable family of vectors $b=\{b_t: b_t \in H_t\}_{t\geq 0}$ satisfying
$$U_{s,t}(b_s \otimes \Omega_t) +U_{s,t}(\Omega_s \otimes b_t) = b_{s+t} ~~\forall s,t\geq0.$$
We say an addit  is centred if $\ip{\Omega_t,b_t}=0$ for all $t\geq0$. 
\end{defn}

Denote the set of all centered addits by $\mathfrak{A}_\Omega(H)$. Every addit $b$ can be written as $b_t=c_t+\lambda t \Omega_t$ such that $c$ is a centered addit and $\lambda \in \C$. 
The following theorem is proved in \cite[Theorem 5.11]{MS1} (see also \cite[Remark 5.12]{MS1}). The second part follows from the construction of  $\Exp_\Omega$ and $\Log_\Omega$.

\begin{thm}\label{explog} Let $(H_t, U_{s,t})$ be a spatial super-product system with canonical unit $\{\Omega_t\}$. There exists a bijection $\Exp_\Omega:  \mathfrak{A}_\Omega(H)  \mapsto \mathfrak{U}_\Omega(H)$ with inverse $\Log_\Omega$ satisfying $$\ip{\Exp_\Omega(b)_t, \Exp_\Omega(b')_t}=e^{\ip{b_1,b'_1}t}~\forall ~b,b' \in \mathfrak{A}_\Omega(H).$$

Moreover $\Exp_\Omega ( \mathfrak{A}_\Omega(H))$ is contained in the product system generated by the centered addits with respect to $\Omega$, and $\Log_\Omega(\mathfrak{U}_\Omega(H))$ is contained in the product system
generated by the units.
\end{thm}

Let $\mathfrak{S}$  be a collection of measurable sections of a super-product system $(H_t, U_{s,t})$. For any fixed $t \in (0,\infty)$, we denote by $H_t^\mathfrak{S}$ the closure of the linear span of the set $\{s^1_{t_1}s^2_{t_2}  
\cdots s^n_{t_n}: \sum_{i=1}^n t_i = t, s^i \in \mathfrak{S}\},$ and we continue to denote the restrictions of $U_{s,t}$ to $H_t^\mathfrak{S}$ by $U_{s,t}$. (Here the product is the image of
$s^1_{t_1}\otimes s^2_{t_2}  \cdots \otimes s^n_{t_n}$ in $H_t$, under the canonical unitary given by 
the associativity axiom.) Then $(H_t^\mathfrak{S}, U_{s,t})$ is the super-product system  generated by the $\mathfrak{S}$.

A super-product system is said to be of type I if units exist and the collection of units generate the super-product system. Thanks to Theorem \ref{explog}, a spatial super-product system with canonical units $\Omega$ is type I if and only if it is generated by $\{\Omega\}\cup \mathfrak{A}_\Omega$. 
It is easily verified that type I super-product systems are indeed  product systems. 
It also follows that the index, defined by Arveson through units, is nothing but the dimension of $\mathfrak{A}_\Omega$, which is a Hilbert space with respect to $\ip{a,b} =\ip{a_1, b_1}$, satisfying $\ip{a_t, b_t}=t\ip{a,b}$ (see  \cite[Lemma 4.8]{MS1}). Further the dimension of $\mathfrak{A}_\Omega$ does not depend on a particular unit $\Omega$ and this index is an invariant. 

The centered addits of CAR product systems are given by one particle vectors $\{\xi \otimes 1_{[0,t]}: \xi \in \k\}$, which generate the product system (\cite[Lemma 7.1]{MS1}). Hence they are of type I and have index $\dim(\k)$.  This shows that $(H^{\k}(t), U_{s,t})$ are non-isomorphic if $\dim(\k)$ varies. Since Arveson systems form a complete invariant  for \en-semigroups on $B(H)$ with respect  to cocycle conjugacy, the following theorem, which was originally proved by Arveson,  implies that any \en-semigroup with a type I product system of index $n$ is cocycle conjugate to the CAR flow of index $n$. Approaching with addits provides  a much simpler proof.

\begin{thm}
For $n \in \overline{\N}$ any type I product system of index $n$ is isomorphic to $(H^{\k}(t), U_{s,t})$ for some  $\k$ satisfying $\dim(k)=n$. 
\end{thm}

\begin{proof} Let $(H_t, U_{s,t})$ be a type I product system of index $n\in \overline{\N}$. Fix a unit $a^0$ and centered addits $\mathfrak{A}_{a^0}$ with orthonormal basis $\{a^i: i \in I\}$. Let $\k$ be a separable Hilbert space of dimension $n$ with orthonormal basis $\{e^i: i \in I\}$.
Set $e^i_{s,t}= 1_{(s,t)}\otimes e^i$ for $s,t \in (0,\infty)$ and $i\in I$.
Define $U_t:H_t\mapsto \Gamma(L^2(0,t)\otimes \k)$ by $U_ta_0 = \Omega$ and $$U_t\left(a^{i_1}_{t_1}\otimes a^{i_2}_{t_2}\otimes \cdots \otimes a^{i_k}_{t_k}\right)=e^{i_1}_{0,t_1}\wedge e^{i_2}_{t_1,t_1+t_2}\wedge \cdots \wedge e^{i_k}_{t_1+\cdots+t_{k-1},t_1+\cdots+t_{k-1}+t_k},$$ where $t_1+t_2+\cdots +t_k=t$, $\{i_1, i_2\cdots i_k\}\neq \{0\}$ and $e^0_{t_i, t_{i+1}}$ means there is no vector in the antisymmetric tensor product. $U_{s,t}$ 
preserves inner products, maps total set of vectors onto a total set, hence extends to a unitary operator, and provides the required isomorphism of  product systems. 
\end{proof}


A spatial super-product system with index $n$ is said to be of type II (or type II$_n$) if units do not generate the super-product system. It is type III if units do not exist. 
Examples of product systems of type II and  III are complicated to construct, but constructing nontrivial examples of super-product systems belonging to those types are readily given by Example \ref{subsemigroup example}. There are three possibilities for $H^{\k}_G(t)$  defined in Example 
\ref{subsemigroup example}  namely
  \begin{itemize}
   \item $G=\N_0$: they are of type I$_{n}$.
   \item $0\in G\neq\N_0$: they are of type II$_{0}$.
   \item $0\notin G$: they are  type III.
  \end{itemize}
  This follows immediately from Lemma 7.1, \cite{MS1}. 
  
If $(H_t, U_{s,t})$ is a super-product system and $(E_t, U_{s,t})$ is a product subsystem 
then  $(E_t^\perp \cap H_t, U_{s,t})$ is a super-product system. So type III super-product systems can be easily constructed from type II super-product systems, by taking the orthogonal complement of the product system generated by units.

\subsection{Cohomology for spatial super-product systems}

Let $(H_t, U_{s,t})$ be a spatial super-product system with a distinguished unit $\Omega$. The embeddings  $$\iota_{s,t}:H_s\mapsto H_{s+t} \qquad \xi\mapsto U_{s,t}(\xi\otimes\Omega_t)$$ allow us to construct an inductive limit of the family of Hilbert spaces $(H_s)_{s>0}$, which we denote by $H_\infty$, together with embeddings  $\iota_s:H_s\to H_\infty$. We will often abuse notation and identify $H_s$ with its corresponding image in $H_\infty$. Notice that each of the vector $\Omega_s$ is mapped to the same element, which we denote by $\Omega_\infty\in H_\infty$.
 
 We can also define a second family of embeddings
 $$\kappa_{s,t}:H_t\mapsto H_{s+t} \qquad \xi\mapsto U_{s,t}(\Omega_s\otimes\xi).$$ Thanks to the associativity axiom,  the squares
 $$ \xymatrix{ H_s \ar[r]^{\iota_{s,t}} \ar[d]_{\kappa_{r,s}} & H_{s+t} \ar[d]^{\kappa_{r,s+t}} \\ H_{r+s} \ar[r]_{\iota_{r+s,t}} & H_{r+s+t}} $$
 commute for all $r,s,t>0$. So there exist isometries $(\kappa_t:H_\infty\mapsto H_\infty)_{t\geq0}$, which define an action of the semigroup $\R_+$ on $H_\infty$, satisfying
 $\kappa_s\iota_t=\iota_{s+t}\kappa_{s,t}$. 
 
 We say a function $f:\R_+^n\to H_\infty \ominus \C\Omega_\infty$ is adapted if  $f(s_1,\ldots,s_n)\in H_{s_1+\cdots+s_n}$ for all $s_1,\ldots,s_n>0$. Let $C^n=C^n(H,\Omega)$ denote the space of all adapted continuous maps $f:\R_+^n\to H_\infty \ominus \C\Omega_\infty$, and $d^n:C^n(H,\Omega)\to C^{n+1}(H,\Omega)$
be defined by
 \begin{align*}  d^nf(s_1,\ldots,s_{n+1})  :=\kappa_{s_1}f(s_2,\ldots,s_{n+1}) &\\ +\sum_{i=1}^n{(-1)^nf(s_1,\ldots,s_i+ s_{i+1},\ldots,s_{n+1})} &+(-1)^{n+1}f(s_1,\ldots,s_n)\end{align*}.
 
 \begin{lem}
$(C, d)$  forms a cochain complex.
 \end{lem}
 
 \begin{proof}
 The map $d^n$ is the restriction of the usual coboundary map in the group cohomology for the action $\R_+ \curvearrowright H_\infty$. We only need to check that the image is adapted, but this follows since $\kappa_s (H_t)\subseteq H_{s+t}$.
 \end{proof}
 
 \begin{defn}
  For all $n\geq0$, the collection of   $n$-cocycles for $(H,\Omega)$ is the space $Z^n(H,\Omega):=\Ker(d^n)$ and the collection of $n$-coboundaries is defined by $B^1(H,\Omega)=0$ and, for $n\geq 2$, $B^n(H,\Omega)=\Ran(d^{n-1})$. The $n$-th cohomology group is the space
  $\H^n(E,\Omega):=Z^n(E,\Omega)/B^n(E,\Omega)$.
 \end{defn}

The cohomology groups are invariant under isomorphisms of super-product systems preserving the canonical unit.
 For all examples we consider in this paper, the spaces of higher cocycles are infinite dimensional and the cohomology groups are difficult to compute. We will instead concentrate on certain distinguished subspaces, which are more tractable. To some extent they measure how far the super-product system is from being a product system.
  
  \begin{defn}
   Let $H=(H_t, U_{s,t})$ be a super-product system with canonical unit $\Omega$. A defective $n$-cochain for $(H,\Omega)$ is a member of $C^n(E,\Omega)$ satisfying
   \begin{align*} a(s_1,\ldots,s_n)\perp \iota_{s_1+\cdots+s_n}U_{s_1,\ldots,s_n}(H_{s_1}\otimes\cdots\otimes H_{s_n})\end{align*}
   for all $s_1,\ldots,s_n>0$, where $U_{s_1,\ldots,s_n}:H_{s_1}\otimes\cdots\otimes H_{s_n}\mapsto H_{s_1,\ldots,s_n}$ is canonical unitary map determined uniquely by the associativity axiom.
   
   We denote the space of defective $n$-cochains by $C^n_{def}(H,\Omega)$  and define, similarly, the collection of defective $n$-cocycles $Z^n_{def}(H,\Omega)$ and coboundaries $B^n_{def}(H,\Omega)$. The corresponding quotient $\H^n_{def}(H,\Omega)$ is the $n$-th defective cohomology group.
  \end{defn}

Clearly for a product system, one always has $C^n_{def}(H,\Omega)=\{0\}$. The following is immediate from definitions.
  
 \begin{prop}
  Let $H=(H_t, U_{s,t})$ and $H=(K_t, U'_{s,t})$  be two spatial super-product systems with canonical units $\Omega$  and $\Omega'$ respectively. Let $V_t:H_t\to K_t$ be an isomorphism of super-product systems taking the unit $(\Omega_t)_{t\geq 0}$ to $(\Omega'_t)_{t\geq 0}$. Then, for each $n\geq 0$, there is an isomorphism
  $ \Phi:C^n_{def}(H,\Omega)\to C^n_{def}(K,\Omega')$ which preserves cocycles and coboundaries. 
  \end{prop}

 One can give an equivalent definition for the cohomology without referring to the inductive limit. For instance, there is a bijective correspondence between the $1-$cocycles and the addits we have defined earlier. Further 
 $\dim Z^1(H,\Omega)$ is the index of $H$, which is independent of the choice of $\Omega$. It is also  easy to see that there is a bijective correspondence with  $2-$cocycles and the $2-$addits defined below.

\begin{defn}
  A $2-$addit for a spatial super-product system $(H_t, U_{s,t})$, with respect to a canonical unit $\{\Omega_t\}$, is a measurable family of vectors $\{a_{s,t}:s, t\geq 0\}$ satisfying
\begin{itemize}
\item[(i)] $a_{s,t} \in H_{s+t}$ $\forall ~s, t\geq 0$,
\item[(ii)] $U_{r+s, t}(a_{r,s}\otimes \Omega_t) +a_{r+s, t}=U_{r, s+t}(\Omega_r \otimes a_{s,t})+a_{r, s+t}$  $\forall ~~r, s,t\geq0$.
\end{itemize}

A $2-$addit is said to be defective if further $a_{s,t} \in \left( U_{s,t}(H_s\otimes H_t)\right)^\perp$  $\forall ~s, t\geq 0$.
\end{defn}

We denote the set of all defective $2-$addits by $\mathfrak{A}^2_\Omega(H)$.  Note that a product system, in particular a type I super-product system, does not admit defective $2-$addits. We can further divide type II$_0$ super-product systems depending upon the existence and abundance of defective $2-$addits.  We say a type II$_0$ super-product system, with a canonical unit $\Omega$, is type II$_0$-I if it is  generated by $\{\Omega\}\cup \mathfrak{A}^2_\Omega$ in the sense, for each $T>0$,  
$H_T$  is closure of the linear span of the set
$$\{a^1_{s_1, t_1} a^2_{s_2, t_2}  \cdots a^n_{s_n, t_n} : ~a^i \in \{\Omega\} \cup \mathfrak{A}^2_\Omega(H),~ i=1,\cdots ,n~\sum^{n}_{i=1} (s_i+t_i)= T \}.$$

\begin{rem}\label{II-I}  For  $\xi, \eta \in \k$, it is easy to check that the family $$a_{\xi, \eta} = \{(1_{0, s}\otimes \xi) \wedge (1_{s, s+t}\otimes \eta):s,t\geq 0\}$$ is a $2-$addit  for the Clifford super-product system $H^{\k}_{2\N_0}$. Further it is a direct verification to see that the subset $\{a_{\xi, \eta}: \xi, \eta \in \k\}$ of  $\mathfrak{A}^2_\Omega$,
together with $\Omega$, generate the Clifford super-product system. In Section \ref{2-addits computation} we get the precise form of a $2-$addit.
\end{rem}

 
\section{CAR flows and even CAR flows}\label{CAR} In this section we recall the \en-semigroups associated with the canonical anti-commutation relations, and compute their super-product systems. 
CAR flows have been already discussed in \cite{Am}, \cite{BISS} and \cite{Bk}. But the problem of showing non-cocycle-conjugacy for CAR flows with different ranks  has still remained open.

Let $K$ be an infinite dimensional separable complex Hilbert space.  
We denote by $\A(K)$ the CAR algebra over $K$, which is the universal $C^*$-algebra generated by 
$\{a(x): x \in K\}$, where $x \mapsto a(x)$ is an antilinear map satisfying the CAR relations:  
\begin{eqnarray}\label{a-relations}
a(x)a(y) +a(y)a(x)&  = & 0, \\ \nonumber
a(x)a(y)^* +a(y)^*a(x) & = & \ip{x,y}1, 
\end{eqnarray} for all $x,y \in K$. \nonumber 

The \textit{quasi-free state} $\omega_A$ on $\A(K)$, associated with a positive contraction $A \in B(K)$, 
is the state determined by its $2n$-point function as  
$$\omega_A(a(x_n) \cdots a(x_1)a(y_1)^* \cdots a(y_m)^* ) = \delta_{n,m} \det (\langle x_i, Ay_j\rangle) ,$$ 
where $\det(\cdot)$ denotes the determinant of a matrix.
Given a positive contraction, it is a fact that such a state always exists and is uniquely determined by the above relation. 

We denote by $(H_A, \pi_A, \Omega_A)$ the GNS triple associated with 
a quasi-free state $\omega_A$ on $\A(K)$, and set $\m_A:=\pi_A(\A(K))''$. 
Every quasi-free state $\omega_A$ of the CAR algebra $\A(K)$ is a factor state, 
(see \cite[Theorem 5.1]{PS}). 
We summarize standard results on the types of von Neumann algebras obtained through the quasi-free representations, in the following theorem. For a proof we refer to  \cite[Lemma 5.3]{PS}, and  for (i) we refer to \cite[Chapter 13]{Arv}, \cite[Section II]{pow2}.
 
\begin{thm}\label{qfstate} Let $K$ be a Hilbert space, let $A\in B(K)$ be a positive contraction.
\begin{itemize}
\item [(i)] $\m_A$ is of type I if and only if $\tr(A -A^2) < \infty.$  
\item [(ii)] $\m_A$ is of type II$_1$ if and only if $A-\frac{1}{2}$ is Hilbert-Schmidt. 
\item [(iii)] $\m_A$ is of type II$_\infty$ if and only if there exists a spectral projection $P$ of $A$, with both $P$ and $1-P$ are of infinite dimensions, $(PAP)^2 -PAP$ is trace-class and $(1-P)A(1-P)- (1-P)/2$ is Hilbert-Schmidt.  
\item [(iv)]   $\m_A$ is of type III otherwise. 
\end{itemize}
\end{thm}

From here onwards we let $K= L^2((0,\infty),\k)$ and $A\in B(K)$ be a positive contraction satisfying  $Ker(A)=Ker(1-A)=\{0\}$. We further assume that $A$ is a Toeplitz operator, meaning $T_t^*AT_t=A$  for all $t\geq 0$. 
Let $j$ be a conjugation on $\k$, and we continue to denote the conjugation on $K$ also as $j$, obtained by $(jf)(s):=jf(s)$  for $f \in K$ for $s\geq0.$ The GNS representation associated with the quasi-free state $\omega_A$ can be concretely realized on the doubled anti-symmetric Fock space $H_A=\Gamma(K)\otimes \Gamma(K)$ by
 $$\pi_A(a(f)) = a(\sqrt{1-A}f)\otimes \Gamma(-1) + 1\otimes a^*(j\sqrt{A}f),$$  with cyclic and separating vector $\Omega_A = \Omega\otimes \Omega$, where $\Omega$ is the vacuum vector of $\Gamma(K)$ and $\Gamma(-1)$ is the second quantization (see \cite{Araki Wyss} and \cite{PS}). Whenever it is convenient we will use this picture. 
 
If $X\in B(K)$ is a positive Toeplitz operator, then $\sqrt{X}T_t\sqrt{X}^{-1}$  extends to an isometry on $H$ (see \cite[Lemma 7.3]{MS2}). We denote the extended semigroup of isometries by $(T^X_t)_{t\geq 0}$. Denote $Y_t=\begin{bmatrix} T_t^{{1-T}} & 0 \\ 0 & jT_t^{{T}}j \end{bmatrix},$ for $t\geq 0$.  Now it follows from \cite[Proposition 2.1.3]{Arv}, there exists a unique \en-semigroup $\{\theta^A_t:t\geq 0\}$ on $B(H_A)$ satisfying  $\theta_t^A(a(f))= a(Y_tf)~~~\forall f \in K\oplus K.$

\begin{ex} 
The restriction of $\theta^A$ to $\m_A$ provides the unique \en-semigroup $\alpha^A=\{\alpha^A_t:t\geq 0\}$ on  $\m_A$, determined by 
$$\alpha_t^A(\pi_A(a(f)))=\pi_A(a(T_tf)),\quad \forall f\in K.$$
We call $\alpha^A$ as the \textit{Toeplitz CAR flow} on $\m_A$ associated with $A$.

Let $\m^e_A$ denotes the von Neumann subalgebra generated by the even products of $\pi_A(a(f)), \pi_A(a^*(g))$. This is the fixed point algebra of the action given by $\pi_A(a(f)) \mapsto -\pi_A(a(f))$. If we assume $tr(A^2-A)=\infty$, then this action is outer and hence $\m_A$ is a factor. The restriction of $\alpha^A$ to $\m_A^e$  is called as the \emph{Toeplitz even CAR flow} associated with $A$ and we denote it by $\beta^A$.
\end{ex}

By construction, all these \en-semigroups has a faithful normal invariant state given by the normal extension of $\omega_A$, which implies the existence of the canonical unit $S=(S_t)_{t\geq 0}$ defined by the isometric extension of $S_tx\Omega=\alpha_t(x)\Omega$ for all $x \in \m_A$, $t\geq 0$. 
The following Proposition characterises when $S$ is a multi-unit.

\begin{prop}\label{Toeplitz and extension property}
For the  \en-semigroups $\alpha^A, \beta^A$,  the canonical unit is a multi-unit if and only if $A=1_{L^2(\R)}\otimes R$ for some $R\in B(\k)$. 
 \end{prop}

  \begin{proof} One way is clear. For the other way, $S$ is a multi-unit if and only if the modular group $(\sigma_s^\Omega)_{s\in \R}$ satisfies $\alpha_t=\sigma_{-s}^{\Omega}\circ\alpha_t\circ\sigma_s^\Omega$ for all $t\geq0$, $s\in\R$(see \cite[Proposition 3.4, (iii)]{MS2}). The modular automorphism $\sigma_s^\Omega$ is the Bogoliubov automorphism associated with $A^{is}(1-A)^{-is}$. Hence  $A^{is}(1-A)^{-is}$ commute with $T_t$ for all $t\geq 0$ and so it is of the form $1_{L^2(\R_+)}\otimes U_s$. By considering the (analytic) generator, we infer that $A(1-A)^{-1}=1_{L^2(\R)}\otimes X$, for some densely defined self-adjoint operator $X$ on $\k$. For  $f \in \Dom((1-A)^{-1})$ we have $$(1-A)^{-1}f=(1+A(1-A)^{-1})f= (1_{L^2(\R_+)}\otimes (1+X))f.$$  This implies that $(1+X)$ has bounded inverse and $$A= 1_{L^2(\R_+)}\otimes \left(1-(1+X)^{-1}\right).$$ The proofs for $\beta^A$  is  similar.  \end{proof}

From here onwards, we assume that $A$ is of the form $1_{L^2(\R_+)}\otimes R$ for some $R\in B(\k)$. Since we have assumed both $A, 1-A$ are injective (in particular $R^2-R\neq 0$), $\tr(A^2-A)=\tr\left(1_{L^2(\R_+)}\otimes (R^2-R)\right) = \infty,$ and hence $\m_A$ is not of type I. In the same way $A-\frac{1}{2}$ is Hilbert-Schmidt if and only $R=\frac{1}{2}$; in that case it is type II$_1$. $\m_A$ can not be of type II$_\infty$, since in that case there will be a non-trivial subspace of $\k$, which is an eigen space for $R$ with eigen value $1$, which will contradict the assumption that $1-A$ is injective. (Even if we relax the condition $1-A$ being  injective, we will only end up with a tensor product of Clifford flow (of even index) with a CAR flow on type I$_\infty$ factor, which are completely classified  by \cite[Theorem 6.1]{MS2}.)
So with our assumptions when $R\neq \frac{1}{2}$, $\m_A$ is of type III. We may replace the suffixes $A$ by $R$ in all our notations 
  and call $\alpha^R, \beta^R$ as just CAR flows, even CAR flows respectively. 


For $m_0\in \m$, it is well-known that the map $m\Omega\mapsto mm_0\Omega$ for all $m\in \m$  is closable; we denote the closure by $\rho_{m_0}$ and its domain by $D(\rho_{m_0})$.  The proof of the first part of the following lemma is exactly same as \cite[Lemma 8.3]{MS1}. The only fact used there, by assuming $\m$ to be a II$_1$  factor, is  $S_t$ being a multi-unit (which was automatic). 

\begin{lem}\label{EtMSt} Let $\alpha$ be an \en-semigroup acting standardly on $\m$ with cyclic and separating vector $\Omega$ satisfying $\ip{\alpha_t(m)\Omega, \Omega}= \ip{m\Omega, \Omega}$ for all $m \in \m$. If further the canonical unit $S=(S_t)_{t\geq 0}$  is a multi-unit,  then 
$$H^\alpha_t= [ \m S_t ] \cap {\alpha_t(\m)}^\prime S_t, $$ where $[\m  S_t ]$
denotes the weak operator closure of $\m  S_t$.

 For any $A=TS_t \in H^\alpha_t$, with $T \in \alpha_t(\m)' $,  $$T\Omega \in D(\rho_{\alpha_t(m)})~ \mbox{and}~ \rho_{\alpha_t(m)}(T\Omega) = T \alpha_t(m)\Omega,~~ \forall~m \in \m.$$
\end{lem}


\begin{proof}
Since $TS_t \in [\m S_t]$, there exists a net $\{m_\lambda\}_{\lambda \in \Lambda}\subseteq \m$ such that $m_\lambda S_t$ converges strongly to $TS_t$. This means $\{m_\lambda\Omega\} \subseteq D(\rho_{\alpha_t(m)})$ converges to $T\Omega$ and $m_\lambda \alpha_t(m)\Omega$ converges to $T \alpha_t(m)\Omega$ for all $m \in \m$, which means the assertion of the lemma.     
\end{proof}

In \cite{Bk}, the super-product systems of CAR flows were computed, assuming few conditions. We provide here a direct proof. 
Let $u_R(f) = \frac{1}{\sqrt{2}}\left(\pi_R(a(f)) + \pi_R(a^*(f))\right)$ for $f \in L^2((0,\infty),\k)$. Then $$u_R(f)u_R(g) + u_R(g)u_R(f) = \mbox{Re}\ip{f,g},~\forall f,g \in L^2((0,\infty),\k),$$ and $\m_R = \{u_R(f) : f \in L^2((0,\infty),\k)\}''.$ 
Since the range of $R^{\frac{1}{2}}$ is dense, we can choose  a total set $\{\xi_k: k \in \N\}\subseteq \k$ such that $\{R^{\frac{1}{2}}\xi_k\}$ forms an orthonormal basis for $\k$. (Choose any countable linearly independent total subset of $\k$, then its image under $R^{\frac{1}{2}}$ is also linearly independent and total. Now use Gram-Schmidt orthogonalization.) If $\{e_n: n \in \N\}$ be any orthonormal basis for $L^2(0,\infty)$, then it is easy to check that the collection $\{u_R(e_{i_1}\otimes \xi_{j_1})u_R(e_{i_2}\otimes \xi_{j_2})\cdots u_R(e_{i_m}\otimes \xi_{j_m})\Omega\}$ with $1\leq i_1<i_2<\cdots <i_m$, $1\leq j_1<j_2<\cdots <j_m$,  $m \in \N_0$ forms an orthonormal basis for the GNS Hilbert space $H_R$ (when $m=0$ we take $\Omega$). Now it is clear that the following choices can be made.
Pick distinct posets $\Lambda_1$, $\Lambda_2$ order isomorphic to $\N$. Let \begin{align*} \p= & \{I=(i_1, i_2 \cdots i_m)\in \Lambda_1^m:~1\leq i_1<i_2<\cdots <i_m, m \in \N_0\};\\ \F= & \{F=(j_1, j_2 \cdots j_m)\in\Lambda_2^m:~1\leq j_1<j_2<\cdots <j_m, m \in \N_0\}.\end{align*}  Choose  $\{f_i\}_{i\in \Lambda_1}\subseteq L^2((0,t),\k)$ and $\{g_j\}_{j \in \Lambda_2}\subseteq  L^2((t,\infty),\k)$, so that 
$\{u_R(I) u_R(F)\Omega: I \in \p, F\in \F\}$ forms an orthonormal basis for the GNS Hilbert space $H_R$, where \begin{align*} u(I) & = u(f_{i_1})u(f_{i_2})\cdots u(f_{i_m}); ~~u(F) =u(g_{j_1})u(g_{j_2})\cdots u(g_{j_m}).\end{align*}



\begin{prop} Let $\alpha^R$ and $ \beta^R$ be respectively the CAR flow and even CAR flow associated with $A=1\otimes R\in B(L^2(\R_+, \k)$. The super-product system associated with $\alpha^R$ and $\beta^R$  are both isomorphic to $E^{\k}_{2\N_0}$. 
\end{prop}

\begin{proof} Since $A$ is of the form $1 \otimes R$, notice that $$\overline{span}\{u_R(f_1)u_R(f_2) \cdots u_R(f_{2n})\Omega: f_i \in L^2((0,t), \k\} = E^{\k}_{2\N_0}(t).$$ For $\xi_t\in E^{\k}_{2\N_0}(t)$, define $T_{\xi_t}(m\Omega) = \alpha_t^R(m) \xi_t$. Then $T_{\xi_t}$ extends to a bounded operator on $H_R=\Gamma(K) \otimes \Gamma(K)$, and satisfies $T_{\xi_t}m = \alpha_t^R(m) T_{\xi_t}$ for all $m \in \m$ (indeed $\alpha_t^R(X) = U^2_{s,\infty}(1 \otimes S_t X S_t^*)(U^2_{s,\infty})^*$ and  $T_{\xi_t}(\xi) = U^2_{s,\infty}(\xi_t\otimes S_t\xi)$). The modular conjugation $J_\Omega$ on $\Gamma(K) \otimes \Gamma(K)$ is the anti-linear extension of $\xi_1\wedge\cdots \xi_n\otimes \eta_1\wedge\cdots \eta_m \mapsto  j\eta_m\wedge\cdots j\eta_1  \otimes   j\xi_n\wedge\cdots j\xi_1 $. It is easy to verify $T_{\xi_t}(m'\Omega) = {\alpha_t^R}'(m') \xi_t$, hence $T_{\xi_t} \in H^\alpha_t$, and the map $\xi_t\mapsto T_{\xi_t}$ is isometric. We prove surjectivity as follows.


Let $TS_t\in H^\alpha_t$, with $T \in \alpha_t(M)'$, and there exists a unique expansion \begin{equation*}\label{ONBexp}T\Omega= \sum_{I \in \p, J\in \F}\lambda(I,F)u_R(I)u_R(F)\Omega, ~~ \lambda(I,F) \in \C.\end{equation*} 
For any $F' \in \F$, thanks to Lemma  \ref{EtMSt}, $T\Omega \in D(\rho_{u_R(F')})$ , and
\begin{align*}
T\Omega &=u_R(F')Tu_R(F')^*\Omega \\
& = u_R(F') \rho_{u_R(F')^*}(T\Omega)~ (\mbox{using Lemma}~\ref{EtMSt})\\
&= \sum_{I \in \p, F\in \F}\lambda(I,F)u_R(F') u_R(I)u_R(F)u_R(F')^*\Omega\\
&=  \sum_{I \in \p, F\in \F}\mu_{F'}(I,F) \lambda(I,F)u_R(I)u_R(F)\Omega,
\end{align*} where $\mu_{F'}(I,F)=(-1)^{\sigma_{F'}(I,F)}$ with $\sigma_{F'}(I,F)= |I||F'|+|F||F'|-|F\cap F'|.$ (In the third line we have used the fact that $\rho_{u_R(F')}$ is closed and the orthogonal sum converges.)   Since the expansion is unique 
we conclude $\lambda(I,F)=0$ except for the terms indexed by $(I,F)$ satisfying $|I|$ is even and $F$ is empty. Hence $T\Omega \in E^{\k}_{2\N_0}(t)$.

To compute the super product systems of the even CAR flows, observe that the GNS Hilbert spaces, $\overline{\{\m^e_R\Omega\}}$ can be identified with subspaces $E^{\k}_{2\N_0}(\infty)$. By restricting $T_{\xi_t}$ defined above, we conclude $E^{\k}_{2\N_0}(t) \subseteq H^\beta_t\Omega$.  To  prove the other inclusion, notice that the collection $\{u_R(I) u_R(F)\Omega: I \in \p, F\in \F\}$,  with either both $I$ and $F$ have even lengths or both have odd lengths, provides an ONB for  $E^{\k}_{2\N_0}(\infty)$. Similar arguments as for CAR flows above, (modified as in the proof of \cite[Proposition 8.16]{MS1-ar}) will complete the proof.
\end{proof}

 \section{2-cocycles}\label{2-addits computation}
 In this section we compute the $2-$addits for the Clifford super-product systems, the family of super-product systems $H^{\k}_{2\N_0}(t)$, 
 discussed in Example \ref{subsemigroup example}.   From this computation, we get back the dimension of the Hilbert space $\k\otimes \k$, as the invariant defined as $2-$index. This would prove that two super-product systems $H^{\k}_{2\N_0}(t)$ and $H^{\k'}_{2\N_0}(t)$ are isomorphic if and only if $\k$ and $\k'$ have same dimension. This also classifies  CAR super-product systems, the family of super-product systems  discussed in Example \ref{subsemigroup example2},  since $E^{ \k}_{2\N_0}(t)$ is isomorphic to $H^{\k\oplus \k}_{2\N_0}(t)$.
 
We identify $L^2([0,t],\k)\opower{n}$ with $L^2([0,t]^n,\k\opower{n})$ by the natural isomorphism. Notice that the subspace $L^2([0,t],\k)^{\wedge n}$ is the collection of functions $f\in L^2([0,t]^n,\k\opower{n})$ satisfying
 \begin{align}\label{flip} f(s_{\sigma(1)},\ldots s_{\sigma(n)})=\epsilon(\sigma) \Pi_\sigma f(s_1,\ldots,s_n),\end{align} for any permutation $\sigma \in S_n$, where $\Pi_\sigma$ is the corresponding tensor flip on $\k\opower{n}$.  We use this identification  and the following identification in all our computations. 

 Notice that $\R^n\ominus\{(x_1, x_2,\ldots , x_n): x_i\neq x_j~\forall i\neq j\}$ has full Lebesgue measure,  hence that in particular $f$ is completely determined by its value on the simplex $\Delta^n_t:=\{(s_1,\ldots,s_n)\in [0,t]^n:~s_1>\cdots>s_n\}$, and  $$\norm{f}^2=\int_{[0,t]^n}\norm{f(\mathbf{s})}^2 d\mathbf{s}=\sum_{\sigma\in\mathfrak{S}_n}\int_{\Delta^n_t}\norm{f(\mathbf{\sigma(\mathbf{s})})}^2d\mathbf{s}=n!\int_{\Delta^n_t}\norm{f(\mathbf{s})}^2d\mathbf{s}.$$ Thus the map $f\mapsto n!f|_{\Delta^n_t}$ induces a Hilbert space isomorphism between $L^2([0,t], \k)^{\wedge n}\to L^2(\Delta^n_t,\k\opower{n})$. 
 This identification is also compatible with right shifts, namely it intertwines the isometries $L^2([0,s], \k)^{\wedge m}\otimes L^2([0,t], \k)^{\wedge n}\to L^2([0,s+t], \k)^{\wedge n+m}$ given by $$f\otimes g\mapsto \left((T_t)\opower{m}f\right)\wedge g$$ with isometries $V_{s,t}:L^2(\Delta^m_s, \k\opower{m})\otimes L^2(\Delta^n_t, \k\opower{n}) \to L^2(\Delta^{m+n}_t, \k\opower{n+m})$  given by $$V_{s,t}(f\otimes g)(s_1,\ldots ,s_{n+m})=f(s_1-t,\ldots,s_m-t)\otimes g(s_{m+1},\ldots,s_{n+m}) $$ when $(s_1,\ldots,s_{n+m})\in [t,s+t]^m\times[0,s]^n$, and otherwise $$V_{s,t}(f\otimes g)(s_1,\ldots ,s_{n+m})=0.$$

Since the second quantization $\Gamma(T_t)$ leaves $n-$particle spaces invariant, projection of a $2-$addit onto an $n-$particle space is again a $2-$addit. So we are basically looking for elements  $\{a(s,t) : s,t \in (0, T)\} \subseteq L^2([0,T], \k)^{\wedge 2n},$
for arbitrarily fixed $n$, satisfying $$a(r,s+t)= a(r+s,t)+ a(r,s) -S_r a(s,t),~~~ \forall ~r,s,t  \in (0, r+s+t),$$ where $S_r$ is the restriction of $\Gamma(T_r)$. We continue to call them as $2-$cocycles. Further appropriate orthogonality conditions should be satisfied for the them to be defective.  Though we are interested  only in the anti-symmetric  case, we prove it in a more general case, which may be useful later.

\begin{lem}\label{2-addit 2particle}
For a defective $2-$cocycle $a(s,t) \in L^2([0,s+t)], \k)\opower{2}$
there exists $f_1 , f_2\in L^2_{\loc}(\R_+;\k\opower{2})$ such that
   $$a(s,t)(x,y)=1_{[s,s+t]\times [0,s]}(x,y) f_1(x-y)+1_{[0,s]\times [s,s+t]}(x,y) f_2(y-x),
  ~~ \forall ~s,t,x, y \geq 0.$$
\end{lem}

\begin{proof}
The $2-$cocycle $\{a(s,t): s, t, \geq 0\}$ is defective means $$a(s,t) \perp U_{s,t} (H_s\otimes H_t) = L^2(\left(([0,s]\times [0,s])\bigcup ([s, s+t]\times [s, s+t])\right), \k\opower{2}),$$  for all $s,t\geq 0$.
So the support of $a(s,t)$ must be contained in $([s,s+t]\times [0,s])\cup([0,s]\times [s,s+t])$. Now, it follows from the relation
   $$a(r,s+t)=a(r+s,t)+a(r,s)- S_ra(s,t),$$
by looking at the various supports, 
   that 
  \begin{align}\label{ast1} a(r,s+t)|_{[r,r+s]\times[0,r]} & = a(r,s)|_{[r,r+s]\times[0,r]}\\ \nonumber  a(r+s,t)|_{[r+s,r+s+t]\times[0,r]} & = a(r,s+t)|_{[r+s,r+s+t]\times[0,r]},\end{align} 
   for all $r,s,t \geq 0$, and also that the same first and second relations hold  true as well in the region $[0,r]\times [r, r+s]$ and $[0,r]\times [r+s,r+s+t]$ respectively. 
   These equations separately imply the existence of  $g_1, g_2\in L^2_{\loc}(\Delta^2_\infty,\k\opower{2})$ which together satisfies
   $$a(s,t)(x,y)=1_{[s,s+t]\times[0,s]}(x,y)g_1(x,y)+ 1_{[0,s]\times [s,s+t]}(x,y)  g_2(y,x),$$ for all $s,t\geq 0$. 
    
   We claim that $g_i(x+s,y+s)=g_i(x,y)$ for all $x,y,s>0$, $i=1,2$; hence $g_1(x,y)=f_1(x-y)$  and $g_2(x,y)=f_2(y-x)$ for some $f_1, f_2\in L^2_{\loc}(\R_+,\k\opower{2})$. Indeed, by picking $(x,y)\in [r+s,r+s+t]\times [r,r+s]$, and evaluating the functions in the following equality in this region
   $$a(r+s,t)+a(r,s)=a(r,s+t)+S_ra(s,t)$$
   at $(x,y)$, one finds that $g_1(x,y)=g_1(x-r,y-r)$, as required. By picking $(x,y)\in [r,r+s]\times [r+s,r+s+t]$, we get the relation for $g_2$ and proof is over.
\end{proof}


We write  $A\leftthreetimes B$, for the simplex $\{(s,t):~s\in A,~t\in B,~s>t\}$. Similarly, $A^{\leftthreetimes n}$ will be used to denote the $A\leftthreetimes A\cdots\leftthreetimes A$, where there are $n$ copies of $A$ in the product

\begin{prop}\label{2-addit Clifford} Let $\{a_{s,t}: s, t \geq 0 \}$ be a defective $2-$addit for $(H^{\k}_{2\N_0}(t), U_{s,t})$. Then $a_{s,t} \in L^2([0,s+t)], \k)^{\wedge2}$.
Further
there exists $f\in L^2_{\loc}(\R_+,\k\opower{2})$ such that
   $$ a_{s,t}(x,y)=1_{[s,s+t]\times [0,s]}(x,y) f(x-y)-1_{[0,s]\times [s,s+t]}(x,y) \Pi_{\k\opower{2}} f(y-x),$$
   for all $s,t,x,y\in (0,\infty)$, where $\Pi_{\k\opower{2}}$ is the usual tensor-flip.
\end{prop} 

\begin{proof}
For $S \subseteq [0,s+t]$, if $P_S$ is the projection in $B(L^2([0,s+t)], \k)\opower{2})$ onto $L^2(S, \k)\opower{2}$, then $P_S\left(L^2([0,s+t, \k)^{\wedge2}\right) =L^2(S, \k)^{\wedge 2}$. Using this, it is easy to verify that a defective $2-$cocycle in $L^2([0,s+t)], \k)^{\wedge2}$ continues to be a defective  cocycle in $L^2([0,s+t)], \k)\opower{2}$. By restricting the case of Lemma \ref{2-addit 2particle} to the antisymmetric subspace, it is immediate from Equation \ref{flip}, that any $2-$addit $a_{s,t}\in L^2([0,s+t)], \k)^{\wedge2}$ is of the form mentioned in the proposition. So we only have to show that there are no defective $2-$cocycles in the $2N$-particle space when $N>1$.

 Pick an integer $N>1$, and let $a(s,t)$ be a defective 2-cocycle in the $2N$-particle subspace $L^2([0,s+t)], \k)^{\wedge 2N}$.  Since $a(s,t)$ is defective, 
  we must have that $a(s,t)$ is orthogonal to 
   \begin{align*} &\bigoplus_{n=0}^{N} L^2([s,s+t]^{\leftthreetimes 2n}, \k\opower{2n})\otimes L^2([0,s]^{\leftthreetimes 2(N-n)}, \k\opower{2(N-n)})\\=& L^2\left( \bigcup_{n=0}^N [s,s+t]^{\leftthreetimes 2n}\times [0,s]^{\leftthreetimes{2(N-n)}} , \k\opower{2N}\right) \end{align*}
where for succinctness, we have abused notation in the cases $n=0,2N$.
Thus, up to a set of measure zero, the support $\Sigma(s,t)$ of $a(s,t)$ satisfies \begin{align*} \Sigma(s,t)&\subseteq\bigcup_{n=1}^N{[s,s+t]^{\leftthreetimes{2n-1}}\times [0,s]^{\leftthreetimes(2(N-n)+1)}}\\ &=[s,s+t]\leftthreetimes\left(\bigcup_{n=1}^N [s,s+t]^{\leftthreetimes2(n-1)}\times[0,s]^{\leftthreetimes2(N-n)}\right)\leftthreetimes[0,s]. \end{align*} 

Let $\Sigma_0(s,t)$ be the collection of points obtained by ignoring the first and last coordinates of points in $\Sigma(s,t)$, i.e.
   $$\Sigma_0(s,t):=\{(s_2,~\ldots,s_{N-1}):~\exists_{s_1,s_N>0}~(s_1,s_2,\ldots,s_{N-1},s_N)\in \Sigma(s,t)\}.$$ We will show that $\Sigma_0(s,t)$ has measure zero. We have $\Sigma_0(s,t)\subseteq A(s,t)$ where we set
   $$A(s,t):=\bigcup_{n=1}^N [s,s+t]^{\leftthreetimes2(n-1)}\times[0,s]^{\leftthreetimes2(N-n)}.$$
   The cocycle identity,
   $a(r,s+t)= a(r+s,t) +a(r,s) -S_r a(s,t),~~\forall r,s,t >0$
  asserts that
   $$\Sigma_0(r,s+t)\subseteq A(r,s+t)\cap \big((A(s,t)+r)\cup A(r+s,t)\cup A(r,s)\big).$$

Now we set \begin{align} B(r,s,t) &:=A(r,s+t)\cap A(r+s,t) \nonumber\\
   &=\bigcup_{n=1}^N [r,r+s+t]^{\leftthreetimes2(n-1)}\times[0,r]^{\leftthreetimes2(N-n)} \nonumber\\
   &\phantom{=}\cap \bigcup_{n=1}^N [r+s,r+s+t]^{\leftthreetimes2(n-1)}\times[0,r+s]^{\leftthreetimes2(N-n)} \nonumber\\
    &=\bigcup_{n=1}^N\bigcup_{k=1}^n{ [r+s,r+s+t]^{\leftthreetimes2(k-1)}\times[r,r+s]^{\leftthreetimes2(n-k)}\times[0,r]^{\leftthreetimes2(N-n)}}. \nonumber
   \end{align}
 Further notice that
$$  (A(s,t)+r)= \bigcup_{n=1}^N [r+s,r+s+t]^{\leftthreetimes2(n-1)}\times[r,r+s]^{\leftthreetimes2(N-n)}  \subseteq A(r+s,t),$$ 
 $$ A(r,s+t) \cap  A(r,s)  = A(r,s) \subseteq  B(r,s,t),    $$
    so we must have $\Sigma_0(r,s+t)\subseteq B(r,s,t)$. Since this holds for all $r,s,t>0$, we have
   $$\Sigma_0(r,u)\subseteq~\cap\{B(r,s,t):~s,t\in(0,\infty),~s+t=u\}.$$
   
   Now if we set
   $$C(n,r,s,t):=\bigcup_{k=1}^n{ [r+s,r+s+t]^{\leftthreetimes2(k-1)}\times[r,r+s]^{\leftthreetimes2(n-k)}\times[0,r]^{\leftthreetimes2(N-n)}}$$
   then, by looking at the components of $[0,r]$, note for any $s_1,s_2,t_1,t_2\in (0,\infty)$ with $s_1+t_1=u=s_2+t_2$, that
   $$C(n,r,s_1,t_1)\cap C(m,r,s_2,t_2)=\emptyset \quad\text{if}\quad n\neq m,$$
   so $\Sigma_0(r,u)$ is contained in
   $$\bigcup_{n=1}^N{\cap\{C(n,r,s,t):~s,t\in(0,\infty),~s+t=u\}}.$$
   But for any $n=2,\ldots,N$ and any $s_1>s_2>\ldots>s_{2N-2}$ such that $(s_1,\ldots,s_{2N-2})\in C(n,r,s,t)$, we have
   $s_1\notin[r,\frac{(s_1+s_2)}{2}] \quad \text{and}\quad s_2\notin [\frac{(s_1+s_2)}{2},r+s+t],$
   so that
   $$(s_1,\ldots,s_{2N-2})\notin C(n,\,r,\,\frac{(s_{1}+s_{2})}{2}-r,\,r+s+t-\frac{(s_{1}+s_{2})}{2});$$
   so  it follows that $\Sigma_0(r,u)\subseteq [0,r]^{\leftthreetimes2(N-1)}$.
   
   We claim that $\Sigma_0(r,u)\subseteq \bigcup_{k=0}^{2^n-1}[\frac{kr}{2^n},\frac{(k+1)r}{2^n}]^{\leftthreetimes2(N-1)}$ for all $n\in \N_0$. We prove the claim by induction. The claim is true when $n=0$. Assuming the claim for $n$ and using
   $$a(r,u)=a(\frac{r}{2},\frac{r}{2}+u)+a(\frac{r}{2},\frac{r}{2})-(T_{\frac{r}{2}})^{2N}a(\frac{r}{2},u),$$
   we obtain $\Sigma_0(r,u)\subseteq \bigcup_{k=0}^{2^{n+1}-1}[\frac{kr}{2^{n+1}},\frac{(k+1)r}{2^{n+1}}]^{\leftthreetimes2(N-1)}$. So $\Sigma_0(r,u)$ must be a set of measure zero.
\end{proof}

The proof of the following Proposition follows  form Proposition \ref{2-addit Clifford} and Lemma \ref{2-addit 2particle}, using the isomorphism between $E^{ \k}_{2\N_0}(t)$ and  $H^{\k\oplus \k}_{2\N_0}(t)$.

\begin{prop}\label{2-addit CAR} Let $\{a_{s,t}: s, t \geq 0 \}$ be a defective $2-$addit for $(E^{ \k}_{2\N_0}(t), U_{s,t})$. Then $a_{s,t}= (a^{1}_{s,t}\otimes \Omega_2) + (\Omega_1\otimes a^{2}_{s,t})+a^{0}_{s,t},$ with  $a^{1}, a^{2}  \in  L^2([0,s+t)], \k)^{\wedge2}$  and $a_{s,t}^{0} \in  L^2([0,s+t)], \k)\otimes L^2([0,s+t)], \k)$, where $\Omega_1$ and $\Omega_2$ are vacuum vectors of the first and second Fock spaces respectively.
Further
there exist $f^{1}, f^{2}, f_1, f_2 \in L^2_{\loc}(\R_+,\k\opower{2})$ such that
   $$ a^{i}_{s,t}(x,y)=1_{[s,s+t]\times [0,s]}(x,y) f^{i}(x-y)-1_{[0,s]\times [s,s+t]}(x,y) \Pi_{\k\opower{2}} f^{i}(y-x),~~ i=1,2 $$
$$ a^{12}_{s,t}(x,y)=1_{[s,s+t]\times [0,s]}(x,y) f_1(x-y) + 1_{[0,s]\times [s,s+t]}(x,y) f_2(y-x),$$ for all $s,t,x,y\in (0,\infty)$, $i=1,2$.
\end{prop}

 \begin{defn}Let $(H_t, U_{s,t})$ be spatial super-product system with canonical unit $\Omega$. We say a $2-$addit $\{a_{s,t}: s,t\geq 0\}$  is orthogonal to another $2-$addit  $\{b_{s,t}: s,t\geq 0\}$ if $a_{s,t}\perp b_{s,t}$ for all $s,t \geq 0$.
 
The $2-$index with respect to $\Omega$ is defined as the supremum of the cardinality of all sets containing mutually orthogonal $2-$addits. 
\end{defn}

Clearly $2-$index is an invariant under isomorphism preserving the canonical unit. If the automorphism group of the super product system acts transitively on the set of all units, then the $2-$index do not depend on a particular unit and it is an invariant for the super-product system. In particular when the unit is unique up to scalars, the $2-$index is an invariant, which is the case for Clifford and CAR super-product systems. 

 For $f\in L^2_{\loc}(\R_+,\k\opower{2})$ we denote the $2-$addit described in Proposition \ref{2-addit Clifford}, by  $a_{s,t}^f$. 

 \begin{lem}\label{local structure lemma}
    For $f,g\in L^2_{\loc}(\R_+,\k\opower{2})$ and $T\in (0,\infty]$,  $a_{s,t}^f\perp a^g_{s,t}$ for all $s,t\in (0,\infty)$ with $s+t\leq T$  if and only if $f( r) \perp g( r)$ for almost all $r\in (0,T)$. 
   \end{lem}
 
 \begin{proof} For $f, g \in L^2_{\loc}(\R_+,\k\opower{2})$, 
 we only have to prove that $$\int_s^{s+t}\int_0^s\ip{f(p-q), g(p-q)}dqdp=0~\forall s,t\in (0,\infty) ~\mbox{with}~ s+t\leq T$$  if and only if $f(r)\perp g(r)$  for almost all $r\in (0,T)$. One way is clear. For the other way, 
fix $\varepsilon>0$. 
For any $(s,t)\in (0,\varepsilon)\times (0,T-\varepsilon)$, we have
    \begin{align*}0 & =\int_s^{s+t}\int_0^s\ip{f(p-q),g(p-q)}dqdp\\ &= - \int_0^t\int_{0}^s\ip{f(u+v), g(u+v)}dvdu ~,    
    \end{align*}
    so that
    $\ip{f(t+s), g(t+s)}=0$ for almost all $(s,t)\in (0,\varepsilon)\times (0,T-\varepsilon)$, for all $\epsilon >0$. That is, $\ip{f( r), g(r )}=0$ for almost all $r\in(0,T)$.
   \end{proof}

 \begin{thm}
 The $2-$index of $(H^{ \k}_{2\N_0}(t), U_{s,t})$ is $n^2$. Consequently the super-product systems $(H^{ \k}_{2\N_0}(t), U_{s,t})$ is isomorphic to $(H^{ \k'}_{2\N_0}(t), U_{s,t})$ if and only if $\dim(\k)=\dim(\k')$. 
 
Also $2-$index of $(E^{\k}_{2\N_0}(t), U_{s,t})$ is $4n^2$. Consequently $(E^{ \k}_{2\N_0}(t), U_{s,t})$ is isomorphic to $(E^{\k'}_{2\N_0}(t), U_{s,t})$ if and only if $\dim(\k)=\dim(\k')$.
 \end{thm}
 
 \begin{proof}
    Pick an orthonormal basis $(e_i)_{i=1}^n$ for $\k$ (possibly with $n=\infty$) and define functions $f_{ij}\in L^2_{\loc}(\R_+,\k\opower{2})$ by
    $ f_{ij}(r):=e_i\otimes e_j$ for all $r\geq0$. 
     Let $a_{ij}$ denote the 2-cocycle with symbol $f_{ij}$, then it is clear that the $a_{ij}$ are pairwise locally orthogonal. Hence the $2-$index  is greater than or equal to $n^2$.
    
    Now assume that we have a set of $n^2$ orthogonal 2-cocycles $a_1,\ldots, a_{n^2}$ with symbols $f_1,\ldots,f_{n^2}$. Then, by Lemma \ref{local structure lemma}, $f_1,\ldots,f_{n^2}$ are orthogonal a.e.
    Thus, for almost all $r\geq0$, $(f_i(r))_{i=1}^{n^2}$ is an ortho-basis for $\k \opower{2}$. If $a$ is a 2-cocycle with symbol $f$ which is orthogonal to $a_1,\ldots,a_{n^2}$ then, again by Lemma \ref{local structure lemma}, we must have $f(r)\perp f_1(r),\ldots,f_{n^2}(r)$ for almost all $r\geq0$, i.e. $f=0$.
   \end{proof}
   
   The above Proposition provides a direct proof that Clifford flows on type II$_1$ factors are non-cocycle if they have different ranks. We further have the following new result for the CAR flows on type III factors.  Notice under the assumption that $R$ and $1-R$ are injective $Rank(R) = dim(\k)$.
    
   \begin{cor}  Let $\alpha^{R_1}, \alpha^{R_2}$ and $ \beta^{R_1}, \beta^{R_2}$ be respectively the CAR flows and even CAR flows associated with $1\otimes R_1, 1\otimes R_2 \in B(L^2(\R_+, \k)$. Then $\alpha^{R_1}$ is not  cocycle conjugate to $\alpha^{R'}$ if  $R_1$ and  $R_2$ have different ranks.  Similarly $\beta^{R_1}$ is also not  cocycle conjugate to $\beta^{R_2}$ if  $R_1$ and  $R_2$ have different ranks.
\end{cor}

  \section{The automorphism group}

To conclude the paper, we calculate the automorphism group of the Clifford super-product system $H^{\k}_{2\N_0}$, which is clearly an invariant of the super-product system.   We show that it is much larger than the gauge group of the corresponding Clifford flow, which was shown in \cite{MS1} to be trivial.

 Let $\mathcal{M}(\R_+;\mathcal{U}(\k\opower{2}))$ denote the group of all measurable, unitary-valued functions endowed with pointwise multiplication.  For a given $\lambda\in\R$ and $F\in\mathcal{M}(\R_+;\mathcal{U}(\k\opower{2}))$ define
  $$U_{(\lambda,F)}(t) \Omega_t=e^{i\lambda t}\Omega_t, \qquad (U_{(\lambda,F)}(t) f)(x,y)=e^{i\lambda t}F(|x-y|)f(x,y) $$ for all $x,y\in [0,t]$, $t>0$, $f\in L^2([0,t];\k)^{\wedge 2}$, where $ \Omega_t$ is the vacuum vector in $H^{\k}_{2\N_0}(t)$.    Now extend $U_{(\lambda,F)}(t)$ to $L^2([0,t];\k)^{\wedge 2n}$ by  $$U_{(\lambda,F)}(t) \left(f_1 \wedge \cdots \wedge  f_n\right) = U_{(\lambda,F)}(t)f_1 \wedge \cdots \wedge  U_{(\lambda,F)}(t) f_n$$ for $f_1, \cdots f_n \in L^2([0,t];\k)^{\wedge 2}$. Clearly $U_{(\lambda,F)}(t)$ extends to a unitary operator on $H^{\k}_{2\N_0}(t)$. 
  
  We denote the automorphism group of $H^{\k}_{2\N_0}$ by $Aut( H^{\k}_{2\N_0})$.
  
  \begin{thm}   $Aut( H^{\k}_{2\N_0})$ is isomorphic to $(\R,+)\times\mathcal{M}(\R_+;\mathcal{U}(\k\opower{2}))$. 
  \end{thm}
  
   \begin{proof}
 It is clear that the family $\{ U_{(\lambda,F)}(t) : t\geq 0\}$ provides an automorphism for any fixed $(\lambda,F)$, and that the map $(\lambda,F) \mapsto  U_{(\lambda,F)}$  induces a homomorphism from $(\R,+)\times\mathcal{M}(\R_+;\mathcal{U}(\k\opower{2}))$ into $Aut( H^{\k}_{2\N_0})$.  To prove injectivity, 
 if the automorphism determined by $(\lambda,F)$ coincides with that of $(\mu,G)$, then the action on the vacuum ensures us that $\lambda=\mu$ and, picking $u,v\in\k \opower{2}$, $T>0$ and setting $$ f=1_{\{0\leq s< t \leq T \}}\otimes (u\otimes v)-1_{\{0\leq t< s\leq T\}}\otimes (v\otimes  u),  $$   we obtain $F(s-t)(u\otimes v)=G(s-t)(u\otimes v)$ for almost all $0\leq t<s\leq T$, but since $u,v$ and $T$ were arbitrary, $F=G$. Thus it remains only to show that the given homomorphism is a surjection. 
   
Since any automorphism $\theta \in Aut( H^{\k}_{2\N_0})$ preserves units, it must satisfy $\theta_t(\Omega_t)=e^{i\lambda t}\Omega_t$ for some $\lambda\in\R$. Thus, by setting  $\theta'_t(x):=e^{-i\lambda t}\theta_t(x)$ for all $t>0$, $x\in  H^\alpha_t$ we obtain an automorphism which preserves the unit $\Omega$.    

Since $\theta'$ preserves defective 2-cocycles, there exists a linear bijection $X$ on $L^2_{\loc}(\R_+;\k \opower{2})$ such that  $ \theta_{s+t}'(a_f(s,t))=a_{Xf}(s,t) $. The equality $\ip{a_f(s,t), a_g(s,t)}=\ip{a_{Xf}(s,t), a_{Xg}(s,t)}$ is equivalent to
   $$\int_0^t\int_0^s\ip{f(p+q), g(p+q)}dqdp=\int_0^t\int_0^s\ip{(Xf)(p+q), (Xg)(p+q)}dqdp$$ for all $s,t>0$ (see the proof of Lemma \ref{local structure lemma}).  Thus $$\ip{f(r),g(r)}=\ip{(Xf)(r), (Xg)(r)}, ~~ \mbox{for almost all}~ r>0.$$  This implies that $X$ restricts to an isometry $\tilde{X}$ on $L^2(\R_+;\k\opower{2})$.  On the other hand $(\theta')^{-1}$ implements the bijection $X^{-1}$, which also restrict to an isometry on $L^2(\R_+;\k\opower{2})$, and hence $\tilde{X}$  is unitary.
   
We claim that $\tilde{X}$ commutes with the orthogonal projections
 $$P_t:L^2(\R_+;\k \opower{2})\to L^2([0,t];\k \opower{2}), 
 ~\qquad f\mapsto f_{[0,t]}$$
  for all $t\geq0$.
 To see this, note that if $f\in\Ker P_T$ then
   $ a_f(s,t) \perp a_g(s,t)$ for all $g\in L^2_{\loc}(\R_+;\k\opower{2})$ and $s+t\leq T$. Hence
 $a_{Xf}(s,t)\perp a_g(s,t)$ for all $g\in L^2_{\loc}(\R_+;\k\opower{2})$ and $s+t\leq T$ which implies, by Lemma \ref{local structure lemma}, that $Xf\in\Ker P_T$, that is $\tilde{X}(1-P_T)H \subseteq  (1-P_T)H$.  Similarly, on the other hand, $\tilde{X^{-1}}(1-P_T)H \subseteq  (1-P_T)H$ and $\tilde{X^{-1}} = \tilde{X}^*$. Hence the claim.   
 Thus, we can identify $\tilde{X}$ with an element of $(L^\infty(\R_+)\otimes 1)'=L^\infty(\R_+;B(\k\opower{2}))$.  Since $X$ is a unitary, it follows that $\tilde{X}$  is given by an $F \in \mathcal{M}(\R_+;\mathcal{U}(\k\opower{2}))$. 
 
 Now since the $2-$addits generate the super-product system (see Remark \ref{II-I}), the automorphism $\theta$ is determined by its action on the $2-$addits and 
 the automorphisms are as claimed.
  \end{proof}

 \begin{rem}
 It is apparent that $Aut( H^{\k}_{2\N_0})$ consists of far more than the restrictions of automorphisms of the Fock product system.  The automorphisms of the Fock product system which leave the even subspaces invariant are all of the form
 $$ \theta_s(\Omega_t)=e^{i\lambda s}\Omega_t,~~ \theta_s (f)=e^{i\lambda s}(1_{L^2([0,t])}\otimes U) f ~~ (s,t>0,~f\in L^2([0,t];\k)),$$
 for some $\lambda\in\R$, $U\in\mathcal{U}(\k$. Thus they form a group isomorphic to $(\R,+)\times\mathcal{U}(\k)$. 
\end{rem}

\end{document}